\documentclass[11pt]{amsart}

\usepackage{amsmath}
\usepackage{amssymb}
\usepackage{amscd}

\begin{document}

\title{Quasipolar Subrings of $3\times 3$ Matrix Rings}

\author{Orhan Gurgun}
\address{Orhan Gurgun, Department of Mathematics, Ankara University,  Turkey}
\email{orhangurgun@gmail.com}

\author{Sait Halicioglu}
\address{Sait Hal\i c\i oglu,  Department of Mathematics, Ankara University, Turkey}
\email{halici@ankara.edu.tr}

\author{Abdullah Harmanci}
\address{Abdullah Harmanci, Department of Maths, Hacettepe University,  Turkey}
\email{harmanci@hacettepe.edu.tr}

\date{\empty}
\date{}
\newtheorem{thm}{Theorem}[section]
\newtheorem{lem}[thm]{Lemma}
\newtheorem{prop}[thm]{Proposition}
\newtheorem{cor}[thm]{Corollary}
\newtheorem{exs}[thm]{Examples}
\newtheorem{defn}[thm]{Definition}
\newtheorem{nota}{Notation}
\newtheorem{rem}[thm]{Remark}
\newtheorem{ex}[thm]{Example}

\maketitle

\begin{abstract} An element $a$ of a ring $R$ is called \emph{quasipolar} provided that there
exists an idempotent $p\in R$ such that $p\in comm^2(a)$, $a+p\in
U(R)$ and $ap\in R^{qnil}$. A ring $R$ is \emph{quasipolar} in
case every element in $R$ is quasipolar. In this paper, we
determine conditions under which subrings of $3\times 3$ matrix
rings over local rings are quasipolar. Namely, if $R$ is a
bleached local ring, then we prove that $\mathcal{T}_3(R)$ is
quasipolar if and only if $R$ is uniquely bleached. Furthermore,
it is shown that $T_n(R)$ is quasipolar if and only if
$T_n\big(R[[x]]\big)$ is quasipolar for any positive integer $n$.
 \\[+2mm]
\end{abstract}

\section{Introduction}
Throughout this paper all rings are associative with identity
unless otherwise stated. Following Koliha and Patricio \cite{KP},
the \emph{commutant} and \emph{double commutant} of an element
$a\in R$ are defined by $comm(a)=\{x\in R~|~xa=ax\}$,
$comm^2(a)=\{x\in R~|~xy=yx~\mbox{for all}~y\in comm(a)\}$,
respectively. If $R^{qnil}=\{a\in R~|~1+ax\in U(R)~\mbox{for
every}~x\in comm(a)\}$ and $a\in R^{qnil}$,  then $a$ is said to
be \emph{quasinilpotent} \cite{H}. An element $a\in R$ is called
\emph{quasipolar} provided that there exists an idempotent $p\in
R$ such that $p\in comm^2(a)$, $a+p\in U(R)$ and $ap\in R^{qnil}$.
A ring $R$ is \emph{quasipolar} in case every element in R is
quasipolar. Properties of quasipolar rings were studied in \cite{CC2, CC1, YC}.
For a ring $R$, let $\mathcal{T}_3(R)=\left\{\left[%
\begin{array}{ccc}
  a_{11} & 0 & 0 \\
  a_{21} & a_{22} & a_{23} \\
  0 & 0 & a_{33} \\
\end{array}%
\right]~|~ a_{11},a_{21},a_{22},a_{23},a_{33}\in R\right\}$. Then
$\mathcal{T}_3(R)$ is a ring under the usual addition and
multiplication, and so $\mathcal{T}_3(R)$ is a subring of
$M_3(R)$. Motivated by results in \cite{C1} and \cite{CC3}, we
study quasipolar subrings of $3\times 3$ matrix
rings over local rings. We prove that $\displaystyle{\left[%
\begin{array}{ccc}
  \Bbb Z_{(2)} & 0 & 0\\
  \Bbb Z_{(2)} & \Bbb Z_{(2)} & \Bbb Z_{(2)} \\
  0 & 0 & \Bbb Z_{(2)} \\
\end{array}%
\right]}$ is quasipolar but the full matrix ring $M_3(\Bbb
Z_{(2)})$ is not quasipolar.

In this paper, $M_n(R)$ and $T_n(R)$ denote the ring of all
$n\times n$ matrices and  the ring of all $n\times n$ upper
triangular matrices over $R$, respectively.  We write $R[[x]]$,
$U(R)$ and $J(R)$ for the power series ring over a ring $R$, the
set of all invertible elements and the Jacobson radical of $R$,
respectively. For $A\in M_n(R)$, $\chi (A)$ stands for the
characteristic polynomial $det( tI_n-A)$.

\section{Quasipolar Elements}

In \cite{N}, Nicholson gives several equivalent characterizations
of strongly clean rings through the endomorphism ring of a module.
Analogously, we present similar results for quasipolar rings. For
convenience, we use left modules and write endomorphisms on the
right. For a module $_RM$, we write $E={End}_R(M)$ for the ring of
endomorphisms of $_RM$.

\begin{lem}\label{lem1}\rm{\cite[Lemma 2]{N} Let $\beta, \pi^2=\pi\in {End}_R(M)$. Then
both $M\pi$ and $M(1-\pi)$ are $\beta$-invariant if and only if
$\pi\beta=\beta\pi$.}
\end{lem}

Similar to \cite[Theorem 2.1]{C3} we have the following results
for quasipolar endomorphisms of a module.

\begin{thm}\label{orhan2}\rm{ Let $\alpha\in E=$ End$_R(M)$. The following are
equivalent.

\begin{enumerate}
    \item $\alpha$ is quasipolar in $E$.
    \item There exists $\pi^2=\pi\in E$ such that $\pi\in
    comm_{E}^2(\alpha)$, $\alpha\pi$ is a unit in $\pi E\pi$ and
    $\alpha(1-\pi)$ is a quasinilpotent in $(1-\pi) E(1-\pi)$.
    \item $M=P\oplus Q$, where $P$ and $Q$ are $\beta$-invariant
    for every $\beta\in comm_E(\alpha)$, $\alpha|_P$ is a unit in
    End$(P)$ and $\alpha|_Q$ is a quasinilpotent in
    End$(Q)$.
    \item $M=P_1\oplus P_2\oplus\cdots\oplus P_n$ for some $n\geq
    1$, where $P_i$ is $\beta$-invariant for every $\beta\in comm_E(\alpha)$, $\alpha|_{P_i}$
    is quasipolar in End$(P_i)$ for each $i$.
\end{enumerate}}
\end{thm}

\begin{proof} $(1)\Rightarrow (2)$ Since $\alpha$ is quasipolar in
$E$, there exists an idempotent $\tau\in R$ such that $\tau\in
comm_E^2(\alpha)$, $\alpha+\tau=\eta \in U(E)$ and $\alpha\tau\in
E^{qnil}$. Let $\pi=1-\tau$. Clearly, $\pi^2=\pi\in
comm_E^2(\alpha)$. Note that $\alpha,\pi,\eta$ and $\tau$ all
commute. Now, multiplying $\alpha+\tau=\eta$ by $\pi$ yields
$\alpha\pi=\eta\pi=\pi\eta\in \pi E \pi$. Since $\eta^{-1}\pi\in
\pi E \pi$ this gives
$(\alpha\pi)(\eta^{-1}\pi)=(\pi\eta)(\eta^{-1}\pi)=\pi$.
Similarly, $(\eta^{-1}\pi)(\alpha\pi)=\pi$ so $\alpha\pi$ is a
unit in $\pi E \pi$. Let $(1-\pi)\gamma (1-\pi)\in comm_{(1-\pi) E
(1-\pi)}(\alpha(1-\pi))$. Then $(1-\pi)\gamma (1-\pi)\in
comm_{E}(\alpha(1-\pi))$. The remaining proof is to show that
$(1-\pi)+\alpha(1-\pi)\gamma(1-\pi)$ is a unit in $(1-\pi) E
(1-\pi)$. Since $\alpha(1-\pi)\in E^{qnil}$,
$1+\alpha(1-\pi)\gamma(1-\pi)$ is a unit in  $E$ and so
$(1-\pi)+\alpha(1-\pi)\gamma(1-\pi)$ is a unit $(1-\pi) E
(1-\pi)$.

$(2)\Rightarrow (3)$ Given $\pi$ as in $(2)$, let $P=M\pi$ and
$Q=M(1-\pi)$. Then $M=P\oplus Q$. For any $\beta\in
comm_E(\alpha)$, the hypothesis $\pi\in comm^2_E(\alpha)$ implies
that $\pi\beta=\beta\pi$. By Lemma~\ref{lem1}, both $P$ and $Q$
are $\beta$-invariant. As in the proof \cite[Theorem 3]{N},
$\alpha\pi=\alpha|_P$ is a unit in End$(P)$. Let $\gamma\in
comm_{End(Q)}(\alpha|_Q)$. We show that $1_Q+\alpha|_Q\gamma$ is a
unit in End$(Q)$. Clearly, $\gamma\in comm_{(1-\pi) E
(1-\pi)}(\alpha(1-\pi))$. Since $\alpha(1-\pi)$ is a
quasinilpotent in $(1-\pi) E(1-\pi)$,
$(1-\pi)+\alpha(1-\pi)\gamma$ is a unit in $(1-\pi) E (1-\pi)$.
Let
$[(1-\pi)+\alpha(1-\pi)\gamma]^{-1}=(1-\pi)\tau(1-\pi)=\tau_0\in
$End$(Q)$ and let $q\in Q$. Then
$(q)[1_Q+\alpha|_Q\gamma]\tau_0=(q+q(1-\pi)\alpha\gamma)\tau_0=
(q(1-\pi)+q\alpha(1-\pi)\gamma)\tau_0=q[(1-\pi)+\alpha(1-\pi)\gamma]\tau_0=(q)1_Q$.
Hence $(1_Q+\alpha|_Q\gamma)\tau_0=1_Q$. Similarly,
$\tau_0(1_Q+\alpha|_Q\gamma)=1_Q$. Thus $\alpha|_Q$ is a
quasinilpotent in End$(Q)$.

$(3)\Rightarrow(4)$ Suppose $M = P \oplus Q$ as in $(3)$. Since
$\alpha|_P$ is a unit in End$(P)$, $\alpha|_P$ is a quasipolar in
End$(P)$ by ~\cite[Example 2.1]{CC2}. As $\alpha|_Q$ is a
quasinilpotent in End$(Q)$, $1_Q+\alpha|_Q$ is a unit in End$(Q)$.
Further, $1_Q^2=1_Q$ and $1_Q\in comm_{End(Q)}^2(\alpha|_Q)$ so
$\alpha|_{Q}$ is quasipolar in End$(Q)$.

$(4)\Rightarrow(1)$ Let $\lambda_i\in $End$(P_i)$. Given the
situation in $(4)$, extend maps $\lambda_i$ in End$(P_i)$ to
$\overline{\lambda_i}$ in End$(M)$ by defining
$(\sum\limits_{j=1}^{n}p_j)\overline{\lambda_i}=(p_i)\lambda_i$
for any $p_j\in P_j$. Then $\overline{\lambda_i}~
\overline{\lambda_j}=0$ if $i\neq j$ while
$\overline{\lambda_i}~\overline{\mu_i}=\overline{\lambda_i\mu_i}$
and
$\overline{\lambda_i}+\overline{\mu_i}=\overline{\lambda_i+\mu_i}$
for all $\mu_i\in $End$(P_i)$. By hypothesis, there exists
$\pi_j^2=\pi_j\in comm_{End(P_j)}^2(\alpha|_{P_j})$, $\sigma_j\in
U\big($End$(P_j)\big)$ such that $\alpha|_{P_j}+\pi_j=\sigma_j$
and $\alpha|_{P_j}\pi_j\in $End$(P_j)^{qnil}$. If
$\pi=\sum\limits_{j=1}^{n}\overline{\pi_j}$ and
$\sigma=\sum\limits_{j=1}^{n}\overline{\sigma_j}$ then
$\pi^2=\sum\limits_{j=1}^{n}\overline{\pi_j}^2=\pi\in $End$(M)$
and $\sigma$ is a unit in $E$ because
$\sigma^{-1}=\sum\limits_{j=1}^{n}\overline{\sigma_j}^{-1}$. Since
$\alpha=\sum\limits_{j=1}^{n}\overline{\alpha|_{P_j}}=\sum\limits_{j=1}^{n}(\overline{-\pi_j+\sigma_j})=-\pi+\sigma$,
we show that $\pi\in comm_E^2(\alpha)$ and $\alpha\pi\in
E^{qnil}$. Since for each $\beta\in comm_E(\alpha)$, $P$ and $Q$
are $\beta$-invariant. Hence, $\pi\beta=\beta\pi$ by
Lemma~\ref{lem1} and so $\pi\in comm_{E}^2(\alpha)$. For any
$\beta\in comm_E(\alpha\pi)$, we only need to show that
$1_E+\beta\alpha\pi$ is an isomorphism in $E$. Note that
$\beta|_{P_j}\in comm_{End(P_j)}(\alpha|_{P_j})$ and
$1_{P_j}+\beta|_{P_j}\alpha|_{P_j}=(\pi+\beta\alpha)|_{P_j}$.
Since $\alpha|_{P_j}\pi_j\in $ End$(P_j)^{qnil}$, $1_{P_j}+
\beta|_{P_j}\alpha|_{P_j}\pi_j=(\pi+\beta\pi\alpha)|_{P_j}$ is a
unit in End$(P_j)$. Let $\gamma_j\in $End$(P_j)$ be such that
$(1_{P_j}+
\beta|_{P_j}\alpha|_{P_j}\pi_j)\gamma_j=1_{P_j}=\gamma_j(1_{P_j}+
\beta|_{P_j}\alpha|_{P_j}\pi_j)$ and let
$m=\sum\limits_{j=1}^{n}p_j$ with $p_j\in P_j$. So
$(\sum\limits_{j=1}^{n}p_j)(1_E+\beta\alpha\pi)\gamma$ =
$\big(\sum\limits_{j=1}^{n}p_j+(\sum\limits_{j=1}^{n}p_j)\beta\alpha\pi\big)\gamma$
=
 $\big(\sum\limits_{j=1}^{n}(p_j)1_{P_j}+(\sum\limits_{j=1}^{n}(p_j)[\beta|_{P_j}\alpha|_{P_j}\pi_j]\big)\gamma$
= $
\big(\sum\limits_{j=1}^{n}(p_j)[1_{P_j}+\beta|_{P_j}\alpha|_{P_j}\pi_j]\big)\gamma$
= $
\big(\sum\limits_{j=1}^{n}(p_j)[1_{P_j}+\beta|_{P_j}\alpha|_{P_j}\pi_j]\gamma_j\big)$
= $ \sum\limits_{j=1}^{n}(p_j)[1_{P_j}]$=$m$ where
$\gamma=\sum\limits_{j=1}^{n}\overline{\gamma_j}$. Similarly, we
have
 $(\sum\limits_{j=1}^{n}p_j)\gamma(1_E+\beta\alpha\pi)=\sum\limits_{j=1}^{n}(p_j)[1_{P_j}]=m$.
Therefore $\alpha\pi\in E^{qnil}$, the proof is completed.
\end{proof}

The following result is a direct consequence of Theorem
\ref{orhan2}.

\begin{cor} \label{orhan3}\rm{ Let $R$ be a ring. The following are
equivalent for $a\in R$.
\begin{itemize}
\item[(1)] $a\in R$ is quasipolar.
\item[(2)] There exists $e^2=e\in R$ such that $e\in comm_R^2(a)$,
$ae\in U(eRe)$ and $a(1-e)\in (1-e)R(1-e)^{qnil}$.
\end{itemize}}
\end{cor}

\section{The Rings $ \mathcal{T}_3(R)$}

For a ring $R$, let $a\in R$, $l_a : R\rightarrow R$ and $r_a :
R\rightarrow R$ denote, respectively, the abelian group
endomorphisms given by $l_a(r) = ar$ and $r_a(r) = ra$ for all
$r\in R$. Thus, for $a$, $b\in R$, $l_a$, $r_b$ is an abelian
group endomorphism such that $(l_a - r_b)(r) = ar - rb$ for any
$r\in R$. A local ring R is called {\it bleached} \cite{C} if, for
any $a\in J(R)$ and any $b\in U(R)$, the abelian group
endomorphisms $l_b - r_a$ and $l_a - r_b$ of $R$ are both
surjective. A local ring R is called {\it uniquely bleached} if,
for any $a\in J(R)$ and any $b\in U(R)$, the abelian group
endomorphisms $l_b - r_a$ and $l_a - r_b$ of $R$ are isomorphic.
According to \cite[Example 2.1.11]{D}, commutative local rings,
division rings, local rings with nil Jacobson radicals, local
rings for which some power of each element of their Jacobson
radicals is central are uniquely bleached. Clearly uniquely
bleached local rings are bleached. But so far it is unknown
whether a bleach local ring is uniquely
bleached. Obviously, $\left[%
\begin{array}{ccc}
  a_{11} & 0 & 0 \\
  a_{21} & a_{22} & a_{23} \\
  0 & 0 & a_{33} \\
\end{array}%
\right]\in U\big(\mathcal{T}_3(R)\big)$ if and only if
$a_{11},a_{22},a_{33}\in U(R)$. Further,
$J\big(\mathcal{T}_3(R)\big)=\left\{\left[%
\begin{array}{ccc}
  a_{11} & 0 & 0 \\
  a_{21} & a_{22} & a_{23} \\
  0 & 0 & a_{33} \\
\end{array}%
\right]~|~ a_{11},a_{22},a_{33}\in J(R), a_{21},a_{23}\in R
\right\}$.  Note that if, for every $A\in \mathcal{T}_3(R)$, there
exists $E^2=E\in comm^2(A)$ such that $A-E\in
U\big(\mathcal{T}_3(R)\big)$ and $EA\in
J\big(\mathcal{T}_3(R)\big)\subseteq \mathcal{T}_3(R)^{qnil}$,
then $-A$ is quasipolar and so $\mathcal{T}_3(R)$ is quasipolar.
 We use this fact in the proof of Theorem \ref{orhan1} without mention.

 By \cite[Example 1]{WC} and \cite[Remark 3.2.11]{Do},
$M_3(R)$ is not quasipolar in general. Our next aim is to
determine to find conditions under which $\mathcal{T}_3(R)$ is
quasipolar. In this direction we can give the following theorem.

\begin{thm}\label{orhan1}\rm{ Let $R$ be a bleached local ring. The following
are equivalent.
\begin{itemize}
\item[(1)] $R$ is uniquely bleached.
\item[(2)] $\mathcal{T}_3(R)$ is quasipolar.
\item[(3)] $T_2(R)$ is quasipolar.
\end{itemize}}
\end{thm}

\begin{proof} $(1)\Rightarrow (2)$ Let $A=\left[%
\begin{array}{ccc}
  a_{11} & 0 & 0 \\
  a_{21} & a_{22} & a_{23} \\
  0 & 0 & a_{33} \\
\end{array}%
\right]\in \mathcal{T}_3(R)$. Consider the following cases.

\textbf{Case 1.} $a_{11},a_{22},a_{33}\in J(R)$. Then $A+I_3\in
U(\mathcal{T}_3(R))$ and $AI_3=A\in
J\big(\mathcal{T}_3(R)\big)\subseteq \mathcal{T}_3(R)^{qnil}$. So
$A$ is quasipolar.

\textbf{Case 2.} $a_{11},a_{22},a_{33}\in U(R)$. Then $A+0\in
U(\mathcal{T}_3(R))$ and $A0=0\in \mathcal{T}_3(R)^{qnil}$. So $A$
is quasipolar.

\textbf{Case 3.} $a_{11}\in U(R),a_{22},a_{33}\in J(R)$. There
exists a unique element $e_{21}\in R$ such that
$a_{22}e_{21}-e_{21}a_{11}=a_{21}$. Let $E=\left[%
\begin{array}{ccc}
  0 & 0 & 0 \\
  e_{21} & 1 & 0 \\
  0 & 0 & 1 \\
\end{array}%
\right]$. Then $E^2=E$, $A-E\in U\big(\mathcal{T}_3(R)\big)$ and
$AE\in J\big(\mathcal{T}_3(R)\big)\subseteq
\mathcal{T}_3(R)^{qnil}$. We show that
$E\in comm^2(A)$. Let $X=\left[%
\begin{array}{ccc}
  x_{11} & 0 & 0 \\
  x_{21} & x_{22} & x_{23} \\
  0 & 0 & x_{33} \\
\end{array}%
\right]\in comm(A)$. Then $XA=AX$ and so

\[
\begin{aligned}
  a_{11}x_{11}=x_{11}a_{11},~
    a_{22}x_{22}=x_{22}a_{22},~a_{33}x_{33}=x_{33}a_{33}\\
x_{21}a_{11}+x_{22}a_{21}=a_{21}x_{11}+a_{22}x_{21}\\
 x_{22}a_{23}+x_{23}a_{33}=a_{22}x_{23}+a_{23}x_{33}
 \end{aligned}
\text{\quad ~~~~~~~~~~\quad}
\begin{gathered}
$(i)$\\
$(ii)$\\
$(iii)$
\end{gathered}\]

\noindent Since $a_{22}e_{21}-e_{21}a_{11}=a_{21}$,
$a_{22}[x_{22}e_{21}-e_{21}x_{11}-x_{21}]-[x_{22}e_{21}-e_{21}x_{11}-x_{21}]a_{11}=0$
by (i) and (ii). By $(1)$, $l_{a_{22}}-r_{a_{11}}$ is injective
and so $x_{22}e_{21}-e_{21}x_{11}=x_{21}$. That is, $XE=EX$. Hence
$E\in comm^2(A)$.

\textbf{Case 4.} $a_{11}\in J(R),a_{22}\in U(R),a_{33}\in J(R)$.
There exist unique elements $e_{21},e_{23}\in R$ such that
$a_{22}e_{21}-e_{21}a_{11}=-a_{21}$ and
$a_{22}e_{23}-e_{23}a_{11}=-a_{23}$. Let $E=\left[%
\begin{array}{ccc}
  1 & 0 & 0 \\
  e_{21} & 0 & e_{23} \\
  0 & 0 & 1 \\
\end{array}%
\right]$. Then $E^2=E$, $A-E\in U\big(\mathcal{T}_3(R)\big)$ and
$AE\in J\big(\mathcal{T}_3(R)\big)\subseteq
\mathcal{T}_3(R)^{qnil}$. We prove $E\in comm^2(A)$. Let $X=\left[%
\begin{array}{ccc}
  x_{11} & 0 & 0 \\
  x_{21} & x_{22} & x_{23} \\
  0 & 0 & x_{33} \\
\end{array}%
\right]\in comm(A)$. Then $XA=AX$. Since
$a_{22}e_{21}-e_{21}a_{11}=-a_{21}$,
$a_{22}[-x_{22}e_{21}+e_{21}x_{11}-x_{21}]-[-x_{22}e_{21}+e_{21}x_{11}-x_{21}]a_{11}=0$
by (i) and (ii). By $(1)$, $l_{a_{22}}-r_{a_{11}}$ is injective
and so $x_{22}e_{21}+x_{21}=e_{21}x_{11}$. Since
$a_{22}e_{23}-e_{23}a_{11}=-a_{23}$,
$a_{22}[-x_{22}e_{23}+e_{23}x_{33}-x_{23}]-[-x_{22}e_{23}+e_{23}x_{33}-x_{23}]a_{11}=0$
by (i) and (iii). By $(1)$, $l_{a_{22}}-r_{a_{11}}$ is injective
and so $x_{22}e_{23}+x_{23}=e_{23}x_{33}$. That is, $XE=EX$. Hence
$E\in comm^2(A)$.

\textbf{Case 5.} $a_{11},a_{22}\in J(R),a_{33}\in U(R)$. There
exists a unique element $e_{23}\in R$ such that
$a_{22}e_{23}-e_{23}a_{33}=a_{23}$. Let $E=\left[%
\begin{array}{ccc}
  1 & 0 & 0 \\
  0 & 1 & e_{23} \\
  0 & 0 & 0 \\
\end{array}%
\right]$. Then $E^2=E$, $A-E\in U\big(\mathcal{T}_3(R)\big)$ and
$AE\in J\big(\mathcal{T}_3(R)\big)\subseteq
\mathcal{T}_3(R)^{qnil}$. We show that
$E\in comm^2(A)$. Let $X=\left[%
\begin{array}{ccc}
  x_{11} & 0 & 0 \\
  x_{21} & x_{22} & x_{23} \\
  0 & 0 & x_{33} \\
\end{array}%
\right]\in comm(A)$. Then $XA=AX$. Since
$a_{22}e_{23}-e_{23}a_{33}=a_{23}$,
$a_{22}[x_{22}e_{23}-e_{23}x_{33}-x_{23}]-[x_{22}e_{23}-e_{23}x_{33}-x_{23}]a_{33}=0$
by (i) and (iii). By $(1)$, $l_{a_{22}}-r_{a_{33}}$ is injective
and so $x_{22}e_{23}-e_{23}x_{33}=x_{23}$. That is, $XE=EX$. Hence
$E\in comm^2(A)$.

\textbf{Case 6.} $a_{11}\in J(R),a_{22},a_{33}\in U(R)$. There
exists a unique element $e_{21}\in R$ such that
$a_{22}e_{21}-e_{21}a_{11}=-a_{21}$. Let $E=\left[%
\begin{array}{ccc}
  1 & 0 & 0 \\
  e_{21} & 0 & 0 \\
  0 & 0 & 0 \\
\end{array}%
\right]$. Then $E^2=E$, $A-E\in U\big(\mathcal{T}_3(R)\big)$ and
$AE\in J\big(\mathcal{T}_3(R)\big)\subseteq
\mathcal{T}_3(R)^{qnil}$. We prove that $E\in comm^2(A)$. Let $X=\left[%
\begin{array}{ccc}
  x_{11} & 0 & 0 \\
  x_{21} & x_{22} & x_{23} \\
  0 & 0 & x_{33} \\
\end{array}%
\right]\in comm(A)$. Then $XA=AX$. Since
$a_{22}e_{21}-e_{21}a_{11}=-a_{21}$,
$a_{22}[-x_{22}e_{21}+e_{21}x_{11}-x_{21}]-[-x_{22}e_{21}+e_{21}x_{11}-x_{21}]a_{11}=0$
by (i) and (ii). By $(1)$, $l_{a_{22}}-r_{a_{11}}$ is injective
and so $x_{22}e_{21}+x_{21}=e_{21}x_{11}$. That is, $XE=EX$. Hence
$E\in comm^2(A)$.

\textbf{Case 7.} $a_{11}\in U(R),a_{22}\in J(R),a_{33}\in U(R)$.
There exist unique elements $e_{21},e_{23}\in R$ such that
$a_{22}e_{21}-e_{21}a_{11}=a_{21}$ and
$a_{22}e_{23}-e_{23}a_{33}=a_{23}$. Let $E=\left[%
\begin{array}{ccc}
  0 & 0 & 0 \\
  e_{21} & 1 & e_{23} \\
  0 & 0 & 0 \\
\end{array}%
\right]$. Then $E^2=E$, $A-E\in U\big(\mathcal{T}_3(R)\big)$ and
$AE\in J\big(\mathcal{T}_3(R)\big)\subseteq
\mathcal{T}_3(R)^{qnil}$. To show
$E\in comm^2(A)$ let $X=\left[%
\begin{array}{ccc}
  x_{11} & 0 & 0 \\
  x_{21} & x_{22} & x_{23} \\
  0 & 0 & x_{33} \\
\end{array}%
\right]\in comm(A)$. Then $XA=AX$. Since
$a_{22}e_{21}-e_{21}a_{11}=a_{21}$,
$a_{22}[x_{22}e_{21}-e_{21}x_{11}-x_{21}]-[x_{22}e_{21}-e_{21}x_{11}-x_{21}]a_{11}=0$
by (i) and (ii). By $(1)$, $l_{a_{22}}-r_{a_{11}}$ is injective
and so $x_{22}e_{21}-e_{21}x_{11}=x_{21}$. Since
$a_{22}e_{23}-e_{23}a_{33}=a_{23}$,
$a_{22}[x_{22}e_{23}-e_{23}x_{33}-x_{23}]-[x_{22}e_{23}-e_{23}x_{33}-x_{23}]a_{33}=0$
by (i) and (iii). By $(1)$, $l_{a_{22}}-r_{a_{33}}$ is injective
and so $x_{22}e_{23}-e_{23}x_{33}=x_{23}$. That is, $XE=EX$. Hence
$E\in comm^2(A)$.

\textbf{Case 8.} $a_{11},a_{22}\in U(R),a_{33}\in J(R)$. There
exists a unique element $e_{23}\in R$ such that
$a_{22}e_{23}-e_{23}a_{33}=-a_{23}$. Let $E=\left[%
\begin{array}{ccc}
  0 & 0 & 0 \\
  0 & 0 & e_{23} \\
  0 & 0 & 1 \\
\end{array}%
\right]$. Then $E^2=E$, $A-E\in U\big(\mathcal{T}_3(R)\big)$ and
$AE\in J\big(\mathcal{T}_3(R)\big)\subseteq
\mathcal{T}_3(R)^{qnil}$. The remaining proof is to show that
$E\in comm^2(A)$. Let $X=\left[%
\begin{array}{ccc}
  x_{11} & 0 & 0 \\
  x_{21} & x_{22} & x_{23} \\
  0 & 0 & x_{33} \\
\end{array}%
\right]\in comm(A)$. Then $XA=AX$. Since
$a_{22}e_{23}-e_{23}a_{33}=-a_{23}$,
$a_{22}[-x_{22}e_{23}+e_{23}x_{33}-x_{23}]-[-x_{22}e_{23}+e_{23}x_{33}-x_{23}]a_{33}=0$
by (i) and (iii). By $(1)$, $l_{a_{22}}-r_{a_{33}}$ is injective
and so $x_{22}e_{23}+x_{23}=e_{23}x_{33}$. That is, $XE=EX$. Hence
$E\in comm^2(A)$.

$(2)\Rightarrow(3)$ Assume that $\mathcal{T}_3(R)$ is quasipolar.
Let $E=\left[%
\begin{array}{ccc}
  1 & 0 & 0 \\
  0 & 1 & 0 \\
  0 & 0 & 0 \\
\end{array}%
\right]\in \mathcal{T}_3(R)$. Then $T_2(R)\cong
E\mathcal{T}_3(R)E$. Thus $T_2(R)$ is quasipolar by
\cite[Proposition 3.6]{YC}.

$(3)\Rightarrow(1)$ It follows from \cite[Proposition 2.9]{CC1}.
\end{proof}

An element $a\in R$ is {\it strongly rad clean} provided that
there exists an idempotent $e\in R$ such that $ae=ea$ and $a-e\in
U(R)$ and $ea\in J(eRe)$. A ring $R$ is {\it strongly rad clean}
in case every element in $R$ is strongly rad clean (cf. \cite{D}).

Due to the proof of Theorem \ref{orhan1}, we have the following.

\begin{cor} \label{orhan4}\rm{ Let $R$ be a local ring. The following
are equivalent.
\begin{itemize}
\item[(1)] $\mathcal{T}_3(R)$ is strongly rad clean.
\item[(2)] $R$ is bleached.
\end{itemize}}
\end{cor}

For a ring $R$, let $\mathcal{L}_3(R)=\left\{\left[%
\begin{array}{ccc}
  a_{11} & 0 & 0 \\
  0 & a_{22} & 0 \\
  a_{31} & 0 & a_{33} \\
\end{array}%
\right]~|~ a_{11},a_{31},a_{22},a_{33}\in R\right\}$. Then
$\mathcal{L}_3(R)$ is a ring under the usual addition and
multiplication, and so $\mathcal{L}_3(R)$ is a subring of
$M_3(R)$. Our next endeavor is to find conditions under which
$\mathcal{L}_3(R)$ is quasipolar.

\begin{prop}\label{orhan5}\rm{ Let $R$ be a bleached local
ring. The following are equivalent.
\begin{itemize}
\item[(1)] $R$ is uniquely bleached.
\item[(2)] $\mathcal{L}_3(R)$ is quasipolar.
\end{itemize}}
\end{prop}

\begin{proof} Let $\varphi: \mathcal{L}_3(R)\rightarrow T_2(R)\oplus
R$ given by $$\left[%
\begin{array}{ccc}
  a_{11} & 0 & 0 \\
  0 & a_{22} & 0 \\
  a_{31} & 0 & a_{33} \\
\end{array}%
\right] \mapsto \bigg(\left[%
\begin{array}{cc}
  a_{33} & a_{31} \\
  0 & a_{11} \\
\end{array}%
\right], a_{22}\bigg).$$ Then $\varphi$ is an isomorphism (see
\cite[Proposition 2.2]{C1}). Since $R$ is local, it is quasipolar.
Hence $\mathcal{L}_3(R)$ is quasipolar if and only if $T_2(R)$ is
quasipolar. Therefore it follows from Theorem~\ref{orhan1}.
\end{proof}

\begin{cor}\label{orhan6} \rm{ Let $R$ be a bleached local
ring. The following are equivalent.
\begin{itemize}
\item[(1)] $R$ is uniquely bleached.
\item[(2)] The ring $\left\{\left[%
\begin{array}{ccc}
  a_{11} & 0 & 0 \\
  0 & a_{22} & 0 \\
  a_{31} & a_{32} & a_{33} \\
\end{array}%
\right]~|~ a_{11},a_{31},a_{32},a_{22},a_{33}\in R\right\}$ is
quasipolar.
\item[(3)] The ring $\left\{\left[%
\begin{array}{ccc}
  a_{11} & 0 & a_{13} \\
  0 & a_{22} & a_{23} \\
  0 & 0 & a_{33} \\
\end{array}%
\right]~|~ a_{11},a_{13},a_{23},a_{22},a_{33}\in R\right\}$ is
quasipolar.
\end{itemize}}
\end{cor}

\begin{proof} $(1)\Leftrightarrow(2)$ Let $$\varphi: \mathcal{T}_3(R)\rightarrow \left\{\left[%
\begin{array}{ccc}
  a_{11} & 0 & 0 \\
  0 & a_{22} & 0 \\
  a_{31} & a_{32} & a_{33} \\
\end{array}%
\right]~|~ a_{11},a_{31},a_{32},a_{22},a_{33}\in R\right\}$$ given
by $A=\left[%
\begin{array}{ccc}
  a_{11} & 0 & 0 \\
  a_{21} & a_{22} & a_{23} \\
  0 & 0 & a_{33} \\
\end{array}%
\right] \mapsto \left[%
\begin{array}{ccc}
  a_{11} & 0 & 0 \\
  0 & a_{33} & 0 \\
  a_{21} & a_{23} & a_{22} \\
\end{array}%
\right]$ for any $A\in \mathcal{T}_3(R)$. Then $\varphi$ is an
isomorphism (see \cite[Corollary 3.4]{C2}). In view of
Theorem~\ref{orhan1}, $\mathcal{T}_3(R)$ is quasipolar if and only
if $\left\{\left[%
\begin{array}{ccc}
  a_{11} & 0 & 0 \\
  0 & a_{22} & 0 \\
  a_{31} & a_{32} & a_{33} \\
\end{array}%
\right]~|~ a_{11},a_{31},a_{32},a_{22},a_{33}\in R\right\}$ is
quasipolar, as asserted.

$(1)\Leftrightarrow(3)$ is symmetric.
\end{proof}

\begin{cor}\label{orhan7} \rm{ Let $R$ be a bleached local
ring. The following are equivalent.
\begin{itemize}
\item[(1)] $R$ is uniquely bleached.
\item[(2)] The ring $S_1=\left\{\left[%
\begin{array}{ccc}
  a_{11} & 0 & a_{13} \\
  0 & a_{22} & 0 \\
  0 & 0 & a_{33} \\
\end{array}%
\right]~|~ a_{11},a_{13},a_{22},a_{33}\in R\right\}$ is
quasipolar.
\item[(3)] The ring $S_2=\left\{\left[%
\begin{array}{ccc}
  a_{11} & 0 & 0 \\
  0 & a_{22} & 0 \\
  0 & a_{32} & a_{33} \\
\end{array}%
\right]~|~ a_{11},a_{31},a_{22},a_{33}\in R\right\}$ is
quasipolar.
\end{itemize}}
\end{cor}

\begin{proof} $(1)\Leftrightarrow(2)$ As in the proof of Proposition~\ref{orhan5},
$R$ is uniquely bleached if and only if $S_1$ is quasipolar, as
asserted.

$(2)\Leftrightarrow(3)$ Let $\varphi: S_1\rightarrow S_2$ given
by $$A=\left[%
\begin{array}{ccc}
  a_{11} & 0 & a_{13} \\
  0 & a_{22} & 0 \\
  0 & 0 & a_{33} \\
\end{array}%
\right] \mapsto \left[%
\begin{array}{ccc}
  a_{22} & 0 & 0 \\
  0 & a_{33} & 0 \\
  0 & a_{13} & a_{11} \\
\end{array}%
\right]$$ for any $A\in S_1$. Then $\varphi$ is an isomorphism.
Hence $S_1$ is quasipolar if and only if $S_2$ is quasipolar.
\end{proof}

Let $R$ be a commutative local ring. By Theorem~\ref{orhan1},
Proposition~\ref{orhan5}, Corollary~\ref{orhan6} and
Corollary~\ref{orhan7}, the rings $$\left[%
\begin{array}{ccc}
  R & 0 & 0\\
  0 & R & 0 \\
  R & 0 & R \\
\end{array}%
\right], \left[%
\begin{array}{ccc}
  R & 0 & 0\\
  R & R & R \\
  0 & 0 & R \\
\end{array}%
\right],\left[%
\begin{array}{ccc}
  R & 0 & R\\
  0 & R & 0 \\
  0 & 0 & R \\
\end{array}%
\right],\left[%
\begin{array}{ccc}
  R & 0 & 0\\
  0 & R & 0 \\
  R & R & R \\
\end{array}%
\right]$$

\noindent are all quasipolar.

\begin{rem}\label{rem1} \rm{ Let $\Bbb Z_{(2)}=\{\frac{m}{n}~|~m,n\in \Bbb Z, 2\nmid
n\}$. By \cite[Example 1]{WC} and \cite[Remark 3.2.11]{Do},
$M_3(\Bbb Z_{(2)})$ is not strongly clean and so it is not
quasipolar. However, by Theorem~\ref{orhan1},
Proposition~\ref{orhan5}, Corollary~\ref{orhan6} and
Corollary~\ref{orhan7}, the rings
$\displaystyle{\left[%
\begin{array}{ccc}
  \Bbb Z_{(2)} & 0 & 0\\
  0 & \Bbb Z_{(2)} & 0 \\
  \Bbb Z_{(2)} & 0 &\Bbb Z_{(2)} \\
\end{array}%
\right]}, \displaystyle{\left[%
\begin{array}{ccc}
  \Bbb Z_{(2)} & 0 & 0\\
  \Bbb Z_{(2)} & \Bbb Z_{(2)} & \Bbb Z_{(2)} \\
  0 & 0 & \Bbb Z_{(2)} \\
\end{array}%
\right]},\displaystyle{\left[%
\begin{array}{ccc}
  \Bbb Z_{(2)} & 0 &\Bbb Z_{(2)}\\
  0 & \Bbb Z_{(2)} & 0 \\
  0 & 0 & \Bbb Z_{(2)} \\
\end{array}%
\right]},\linebreak \displaystyle{\left[%
\begin{array}{ccc}
  \Bbb Z_{(2)} & 0 & 0\\
  0 & \Bbb Z_{(2)} & 0 \\
  \Bbb Z_{(2)} & \Bbb Z_{(2)} & \Bbb Z_{(2)} \\
\end{array}%
\right]}$ are all quasipolar.}
\end{rem}

\section{Matrices Over Power Series Rings}

In this section, we characterize quasipolar matrices over the
power series ring of a local ring. In order to prove
Theorem~\ref{orhan11}, we need the following lemma.

\begin{lem}\label{lem2} \rm{ Let $R$ be a commutative local ring and $A(x)\in
M_2\big(R[[x]]\big)$. The following are equivalent.
\begin{itemize}
\item[(1)] $\chi\big(A(0)\big)$ has a root in $J(R)$ and a root in
$U(R)$.
\item[(2)]$\chi\big(A(x)\big)$ has a root in $J\big(R[[x]]\big)$ and a root in
$U\big(R[[x]]\big)$.
\end{itemize}}
\end{lem}

\begin{proof} $(1)\Rightarrow (2)$ Assume that $\chi\big(A(0)\big)=y^2-\mu y-\lambda$ has a root $\alpha\in J(R)$ and a root
$\beta\in U(R)$. Let $y=\sum\limits_{i=0}^{\infty}b_ix^i$. Then
$y^2=\sum\limits_{i=0}^{\infty}c_ix^i$ where
$c_i=\sum\limits_{k=0}^{i}b_kb_{i-k}$. Let
$\mu(x)=\sum\limits_{i=0}^{\infty}\mu_ix^i,
\lambda(x)=\sum\limits_{i=0}^{\infty}\lambda_ix^i\in R[[x]]$ where
$\mu_0=\mu$ and $\lambda_0=\lambda$. Then, $y^2-\mu (x)y-\lambda
(x)=0$ holds in $R[[x]]$ if the following equations are satisfied:
$$\begin{array}{c}
b_0^2-b_0\mu_0-\lambda_0=0;\\
(b_0b_1+b_1b_0)-(b_0\mu_1+b_1\mu_0)-\lambda_1=0;\\
(b_0b_2+b_1^2+b_2b_0)-(b_0\mu_2+b_1\mu_1+b_2\mu_0)-\lambda_2=0;\\
\vdots
\end{array}$$ Obviously, $\mu_0=trA(0)=\alpha+\beta\in U(R)$. Let $b_0=\alpha$. Since
$R$ is commutative local, there exists some $b_1\in R$ such that
$$b_0b_1+b_1(b_0-\mu_0)=\lambda_1+b_0\mu_1.$$ Further, there exists some $b_2\in R$ such that
$$b_0b_2+b_2(b_0-\mu_0)=\lambda_2-b_1^2+b_0\mu_2+b_1\mu_1.$$ By
iteration of this process, we get $b_3,b_4,\cdots $. Then $y^2-\mu
(x)y -\lambda (x)=0$ has a root $\alpha(x)\in J\big(R[[x]]\big)$.
If $b_0=\beta$, analogously, we show that $y^2-\mu (x)y -\lambda
(x)=0$ has a root $\beta(x)\in U\big(R[[x]]\big)$.

$(2)\Rightarrow (1)$ Suppose that $\chi\big(A(x)\big)=y^2-\mu(x)
y-\lambda(x)$ has a root $\alpha(x)\in J\big(R[[x]]\big)$ and a
root $\beta(x)\in U\big(R[[x]]\big)$.  Then $\mu(x)=trA(x)$ and
$-\lambda(x)=detA(x)$. Hence $\mu(0)=trA(0)$ and
$-\lambda(0)=detA(0)$. Thus, $\chi\big(A(0)\big)=y^2-\mu(0)
y-\lambda(0)$. Since $\alpha(x)^2-\mu(x) \alpha(x)-\lambda(x)=0$
and $\beta(x)^2-\mu(x) \beta(x)-\lambda(x)=0$, $\alpha(0)^2-\mu(0)
\alpha(0)-\lambda(0)=0$ and $\beta(0)^2-\mu(0)
\beta(0)-\lambda(0)=0$. Then $\chi\big(A(0)\big)=y^2-\mu(0)
y-\lambda(0)$ has a root $\alpha(0)\in J(R)$ and a root
$\beta(0)\in U(R)$.
\end{proof}

\begin{thm} \label{orhan11}\rm{Let $R$ be a commutative local ring. The following are equivalent.
\begin{itemize}
\item[(1)] $A(0)\in M_2(R)$ is quasipolar.
\item[(2)] $A(x)\in
M_2\big(R[[x]]\big)$ is quasipolar.
\end{itemize}}
\end{thm}

\begin{proof} $(1)\Rightarrow(2)$ It is known that $R[[x]]$ is
local. To complete the proof we consider the following cases:

\begin{itemize}
    \item[(i)] $A(0)\in GL_2(R)$,
    \item[(ii)] $detA(0), trA(0)\in J(R)$,
    \item[(iii)] $detA(0)\in J(R),
trA(0)\in U(R)$ and $\chi\big(A(0)\big)$ is solvable in $R$.
\end{itemize}

If $A(0)\in GL_2(R)$, then $A(x)\in GL_2\big(R[[x]]\big)$ and so
$A(x)\in M_2\big(R[[x]]\big)$ is quasipolar by \cite[Example
2.1]{CC2}. If $detA(0)$, $trA(0)\in J(R)$, then $trA(x)$,
$detA(x)\in J\big(R[[x]]\big)$ and so $A(x)$ is quasipolar by
\cite[Theorem 2.6]{CC2}. Now suppose that $detA(0)\in J(R),
trA(0)\in U(R)$ and $\chi\big(A(0)\big)$ has two roots
$\alpha,\beta\in R$. Then $detA(x)\in J\big(R[[x]]\big)$ and
$trA(x)\in U\big(R[[x]]\big)$. Since $detA(0)\in J(R)$ and
$trA(0)\in U(R)$, either $\alpha\in J(R)$ or $\beta\in J(R)$.
Without loss of generality, we assume that $\alpha\in J(R)$ and
$\beta\in U(R)$. According to Lemma~\ref{lem2},
$\chi\big(A(x)\big)$ has a root in $J\big(R[[x]]\big)$ and a root
in $U\big(R[[x]]\big)$. Hence $A(x)$ is quasipolar in
$M_2\big(R[[x]]\big)$ by \cite[Proposition 2.8]{CC2}.

$(2)\Rightarrow(1)$ is similar to the proof of
$(1)\Rightarrow(2)$.
\end{proof}

\begin{ex}\rm{Let $R = {\Bbb
Z}_4[[x]]$, and let $$A(x)=\left[
\begin{array}{cc}
\overline{0}&-\sum\limits_{n=1}^{\infty}(\overline{1}+\overline{3}^n)x^n\\
\overline{1}&\overline{3}-
\sum\limits_{n=1}^{\infty}(\overline{1}+\overline{3}^n)x^n\end{array}
\right]\in M_2(R).$$ Obviously, ${\Bbb Z}_4$ is a commutative
local ring. Since $A(0)=\left[
\begin{array}{cc}
\overline{0}&\overline{0}\\
\overline{1}&\overline{3} \end{array} \right]$, \linebreak
$\chi\big(A(0)\big)=t^2-trA(0)t+detA(0)=t^2-\overline{3}t=t(t-\overline{3})$
is solvable in $\Bbb Z_4$. By \cite[Proposition 2.8]{CC2},
$A(0)\in M_2(\Bbb Z_4)$ is quasipolar. In view of
Theorem~\ref{orhan11}, $A(x)\in M_2(R)$ is quasipolar.}
\end{ex}

\begin{thm}\label{orhan12}\rm{Let $R$ be a commutative
local ring and for $m\geq 1$ \linebreak $A(x)\in
M_2\big(R[[x]]/(x^{m})\big)$. The following are equivalent.
\begin{itemize}
\item[(1)] $A(0)\in M_2(R)$ is quasipolar.
\item[(2)] $A(x)\in
M_2\big(R[[x]]/(x^{m})\big)$ is quasipolar.
\end{itemize}}
\end{thm}

\begin{proof} The proof is similar to that of
Theorem~\ref{orhan11}.
\end{proof}

\begin{ex}\rm{Let $R = {\Bbb
Z}_4[[x]]/(x^2)$, and let $$\overline{A}(x)=\left[
\begin{array}{cc}
\overline{3}+(x^2)&\overline{2}+\overline{2}x+(x^2)\\
\overline{2}+x +(x^2)&\overline{2}+\overline{3}x+(x^2) \end{array}
\right]\in M_2(R).$$ Obviously, ${\Bbb Z}_4$ is a commutative
local ring. Since $A(0)=\left[
\begin{array}{cc}
\overline{3}&\overline{2}\\
\overline{2}&\overline{2} \end{array} \right]$,
$\chi\big(A(0)\big)=t^2-trA(0)t+detA(0)=t^2-t+\overline{2}=(t-\overline{3})(t+\overline{2})$
is solvable in $\Bbb Z_4$. By \cite[Proposition 2.8]{CC2},
$A(0)\in M_2(\Bbb Z_4)$ is quasipolar. In view of
Theorem~\ref{orhan12}, $A(x)\in M_2(R)$ is quasipolar.}
\end{ex}

\begin{lem}\label{lem3} \rm{ Let $R$ be a local ring. Then $R$ is uniquely
bleached if and only if $R[[x]]$ is uniquely bleached.}
\end{lem}

\begin{proof} Assume that $R$ is uniquely
bleached. Then $l_u-r_j$ is an isomorphism for any $j\in J(R)$ and
$u\in U(R)$ and let $f(x)=\sum\limits_{i=1}^{\infty}a_ix^i\in
R[[x]]$. Since $R$ is bleached, by \cite[Example 2.1.11(6)]{D},
$R[[x]]$ is bleached. If, for
$j(x)=\sum\limits_{i=1}^{\infty}j_ix^i\in J\big(R[[x]]\big)$ and
$u(x)=\sum\limits_{i=1}^{\infty}u_ix^i\in U\big(R[[x]]\big)$,
$(l_{j(x)}-r_{u(x)})(f(x))=0$, then

\[
\begin{aligned}
 j_0a_0&=a_0u_0\\
j_0a_1+j_1a_0&=a_0u_1+a_1u_0\\
j_0a_2+j_1a_1+j_2a_0&=a_0u_2+a_1u_1+a_2u_0\\
&\vdots\\
\end{aligned}
\text{\quad ~~~~~~~~~~\quad}
\begin{gathered}
(i_1)\\
(i_2)\\
(i_3)\\
\vdots\\
\end{gathered}
\]

\noindent By assumption, $l_{j_0}-r_{u_0}$ is an isomorphism and
so $a_0=0$ by $(i_1)$. As $a_0=0$, by $(i_2)$, $j_0a_1=a_1u_0$ and
so $a_1=0$ by assumption. Since $a_0=0=a_1$, by $(i_3)$,
$j_0a_2=a_2u_0$ and so $a_2=0$ by assumption. By iteration of this
process, we deduce that $f(x)=0$. Hence $l_{j(x)}-r_{u(x)}$ is an
isomorphism and so $R[[x]]$ is uniquely bleached. Conversely,
suppose that $R[[x]]$ is uniquely bleached. Then
$l_{j(x)}-r_{u(x)}$ is an isomorphism for any
$j(x)=\sum\limits_{i=1}^{\infty}j_ix^i\in J\big(R[[x]]\big)$ and
$u(x)=\sum\limits_{i=1}^{\infty}u_ix^i\in U\big(R[[x]]\big)$ and
let $r\in R$. Let $(l_{j}-r_{u})(r)=0$ with $j\in J(R)$ and $u\in
U(R)$. Since $j\in J\big(R[[x]]\big)$ and $u\in
U\big(R[[x]]\big)$, by assumption, $r=0$ and so $l_{j}-r_{u}$ is
injective. The remaining proof is to show that $l_{j}-r_{u}$ is
surjective. Since $l_{j(x)}-r_{u(x)}$ is an isomorphism where
$j(0)=j$ and $u(0)=u$, for any $r\in R$, we can find some
$f(x)=\sum\limits_{i=1}^{\infty}a_ix^i\in R[[x]]$ such that
$j(x)f(x)-f(x)u(x)=r$. Hence $ja_0-a_0u=r$ with $a_0\in R$ and so
$l_{j}-r_{u}$ is surjective. Thus $R$ is uniquely bleached.
\end{proof}

\begin{prop}\label{orhan9} \rm{ Let $R$ be a bleached local
ring. The following are equivalent.
\begin{itemize}
\item[(1)] $\mathcal{T}_3(R)$ is quasipolar.
\item[(2)] $\mathcal{T}_3\big(R[[x]]\big)$ is quasipolar.
\end{itemize}}
\end{prop}

\begin{proof} $(1)\Rightarrow (2)$ Assume that $\mathcal{T}_3(R)$ is quasipolar. By Theorem~\ref{orhan1}, $R$ is uniquely bleached.
 Note that if $R$ is local, then so is $R[[x]]$ because
$R/J(R)\cong R[[x]]/J\big(R[[x]]\big)$. According to
Lemma~\ref{lem3}, $R[[x]]$ is uniquely bleached. Hence
$\mathcal{T}_3\big(R[[x]]\big)$ is quasipolar by
Theorem~\ref{orhan1}.

$(2)\Rightarrow (1)$ Suppose that $\mathcal{T}_3\big(R[[x]]\big)$
is quasipolar. Then $R[[x]]$ is uniquely bleached by
Theorem~\ref{orhan1}. In view of Lemma~\ref{lem3}, $R$ is uniquely
bleached. Hence $\mathcal{T}_3(R)$ is quasipolar by
Theorem~\ref{orhan1}.
\end{proof}

\begin{cor}\label{orhan10} \rm{ Let $R$ be a bleached local
ring. For any positive integer $n$, the following are equivalent.
\begin{itemize}
\item[(1)] $T_n(R)$ is quasipolar.
\item[(2)] $T_n\big(R[[x]]\big)$ is quasipolar.
\end{itemize}}
\end{cor}

\begin{proof} By \cite[Proposition 2.9]{CC1} and
Lemma~\ref{lem3}, the proof is completed.
\end{proof}

\end{document}